\documentclass[12pt]{article}

\usepackage{amsmath, epsfig, cite}
\usepackage{amssymb}
\usepackage{amsfonts}
\usepackage{latexsym}
\usepackage{float}
\usepackage{color}
\usepackage{booktabs}
\usepackage{graphicx}
\usepackage{caption}
\usepackage{mathtools}
\usepackage[colorlinks=true,linkcolor=blue,citecolor=blue,urlcolor=blue]{hyperref}

\newtheorem{thm}{Theorem}

\newtheorem{lem}{Lemma}[section]

\newcommand{\pf}{\noindent{\it Proof.} }

\def\Z{{\mathbb Z}}
\def\Q{{\mathbb Q}}

\numberwithin{equation}{section}
\allowdisplaybreaks

\newcommand{\qed}{\hfill$\square$\medskip}

\begin{document}

\begin{center}
{\Large\bf Two $q$-congruences from Jackson's ${}_8\phi_7$ summation and Andrews' ${}_4\phi_3$ summation}
\end{center}

\vskip 2mm \centerline{Ji-Cai Liu and Qing-Yuan Tao}
\begin{center}
{\footnotesize Department of Mathematics, Wenzhou University, Wenzhou 325035, PR China\\[5pt]
{\tt jcliu2016@gmail.com, 15514228254@163.com} \\[10pt]
}
\end{center}

\vskip 0.7cm \noindent{\bf Abstract.}
We prove two $q$-congruence conjectures of Guo on truncated basic hypergeometric series.  The first result strengthens a congruence obtained from Jackson's terminating very-well-poised ${}_8\phi_7$ summation from the modulus $\Phi_n(q)^4$ to the modulus $\Phi_n(q)^5$ when $n\equiv3\pmod 5$.  The second result proves a cyclotomic congruence modulo $\Phi_n(q)$ when $n\equiv1\pmod4$, by a specialization of Andrews' terminating ${}_4\phi_3$ summation.

\vskip 3mm \noindent {\it Keywords}: $q$-supercongruences; cyclotomic polynomials; Jackson's ${}_8\phi_7$ summation; Andrews' ${}_4\phi_3$ summation.

\vskip 2mm
\noindent{\it MR Subject Classifications}: 33D15; 11A07; 11B65.

\section{Introduction}

The study of supercongruences for truncated hypergeometric series has been closely connected with arithmetic geometry, $p$-adic special functions, and Gaussian hypergeometric functions.  Deines, Fuselier, Long, Swisher and Tu~\cite{DeinesEtAl} related certain truncated hypergeometric series to higher dimensional analogues of Legendre curves, and many subsequent works have investigated stronger congruences and their $q$-analogues (see, for example, \cite{Guo2023,GuoZudilin,PanTaurasoWang,WangPan}).  In the $q$-setting, congruences modulo powers of cyclotomic polynomials often refine classical congruences modulo powers of primes after the specialization $q\to1$.

We use the standard notation
\begin{equation*}
(a;q)_k=(1-a)(1-aq)\cdots(1-aq^{k-1}),\qquad
(a_1,\ldots,a_r;q)_k=\prod_{i=1}^r(a_i;q)_k,
\end{equation*}
and
\begin{equation*}
[N]=[N]_q=\frac{1-q^N}{1-q}.
\end{equation*}
The $n$-th cyclotomic polynomial is defined by
\[
  \Phi_n(q)=\prod_{\substack{1\le j\le n\\ \gcd(j,n)=1}}(q-\zeta^j),
\]
where $\zeta$ is a primitive $n$-th root of unity.
If $F(q)$ and $G(q)$ are rational functions over $\Z[q]$, then
\begin{equation*}
F(q)\equiv G(q)\pmod{\Phi_n(q)^r}
\end{equation*}
means that $F(q)-G(q)$ can be written as $U(q)/V(q)$, where
\begin{equation*}
U(q),V(q)\in\Z[q],\qquad \gcd(V(q),\Phi_n(q))=1,\qquad \Phi_n(q)^r\mid U(q).
\end{equation*}
The same convention is used for rational functions involving auxiliary parameters.

Guo and Zudilin~\cite{GuoZudilin} introduced the method of creative microscoping, which has become an effective tool for proving $q$-supercongruences.  Guo~\cite{Guo2025} subsequently used Jackson's terminating very-well-poised ${}_8\phi_7$ summation and creative microscoping to prove several new congruences, and left two conjectures open.  More precisely, for $n\equiv3\pmod 5$ and $m=(4n-2)/5$, Guo proved the $q$-congruence
\[
\sum_{k=0}^{n-1}
\frac{(1+q^{5k+1})(q^2;q^5)_k^5}{(1+q)(q^5;q^5)_k^5}q^{5k}
\equiv
\frac{(q^3,q^4,q^4,q^7;q^5)_m}{(q^2,q^5,q^5,q^6;q^5)_m}
\left(1+\frac{[4n]^2q^{-n}}{[3n]^2}\right)
\pmod{\Phi_n(q)^4}.
\]
Guo~\cite[Conjecture 6.1]{Guo2025} asserts that the same congruence remains valid modulo $\Phi_n(q)^5$.  The second conjecture \cite[Conjecture 6.2]{Guo2025} is a congruence modulo $\Phi_n(q)$ related to the truncated hypergeometric series
\[
{}_4F_3\!\binom{1/4,1/4,1/4,1/4}{1,1,1}.
\]

The purpose of this paper is to prove both conjectures \cite[Conjectures 6.1 and 6.2]{Guo2025}.
\begin{thm}\label{thm:conj61}
Let $n$ be a positive integer with $n\equiv3\pmod 5$.
Then
\begin{align}\label{eq:theorem-one}
&\sum_{k=0}^{n-1}
\frac{(1+q^{5k+1})(q^2;q^5)_k^5}{(1+q)(q^5;q^5)_k^5}q^{5k} \notag\\
&\quad\equiv
\frac{(q^3,q^4,q^4,q^7;q^5)_{(4n-2)/5}}{(q^2,q^5,q^5,q^6;q^5)_{(4n-2)/5}}
\left(1+\frac{[4n]^2q^{-n}}{[3n]^2}\right)
\pmod{\Phi_n(q)^5}.
\end{align}
\end{thm}

\begin{thm}\label{thm:conj62}
Let $n>1$ be an integer with $n\equiv1\pmod4$.  Then
\begin{equation}\label{eq:theorem-two}
\sum_{k=0}^{n-1}\frac{(q;q^4)_k^4}{(q^4;q^4)_k^4}q^{4k}
\equiv
q^{(n-1)/4}
\frac{(q;q^2)_{(n-1)/4}^2(q^2;q^4)_{(n-1)/4}}
{(q^2;q^2)_{(n-1)/4}^2(q^4;q^4)_{(n-1)/4}}
\pmod{\Phi_n(q)}.
\end{equation}
\end{thm}

The proof of Theorem~\ref{thm:conj61} occupies Section~\ref{sec:proof-one}.  Its main point is a regularized root-of-unity specialization of Jackson's summation, which supplies the missing congruence modulo $\Phi_n(q)$ in Guo's two-parameter microscoping congruence.  The proof of Theorem~\ref{thm:conj62} is given in Section~\ref{sec:proof-two}; it is based on Andrews' terminating ${}_4\phi_3$ summation.

\section{Proof of Theorem \ref{thm:conj61}}\label{sec:proof-one}

We shall use Jackson's terminating very-well-poised ${}_8\phi_7$ summation in the following form (see Gasper and Rahman~\cite{GasperRahman}).

\begin{lem}\label{lem:jackson}
If $N$ is a nonnegative integer and
\begin{equation*}
x^2q=yzuvq^{-N},
\end{equation*}
then
\begin{align}\label{eq:jackson}
&\sum_{k=0}^{N}
\frac{1-xq^{2k}}{1-x}
\frac{(x,y,z,u,v,q^{-N};q)_k}
     {(q,xq/y,xq/z,xq/u,xq/v,xq^{N+1};q)_k}q^k   \notag\\
&\qquad =
\frac{(xq,xq/yz,xq/yu,xq/zu;q)_N}
     {(xq/y,xq/z,xq/u,xq/yzu;q)_N}.
\end{align}
\end{lem}

The following lemma is the key additional congruence needed for Guo's microscoping argument.  Its proof keeps an auxiliary leading parameter until the removable singularity in the middle summand has been resolved.

\begin{lem}\label{lem:paramzero}
Let $n\equiv3\pmod5$. For indeterminates $a$ and $b$, define
\begin{equation}\label{eq:paramS}
\mathcal{S}_n(a,b)=
\sum_{k=0}^{n-1}
\frac{(1+q^{5k+1})(aq^2,q^2/a,bq^2,q^2/b,q^2;q^5)_k}
     {(1+q)(aq^5,q^5/a,bq^5,q^5/b,q^5;q^5)_k}q^{5k}.
\end{equation}
Then
\begin{equation}\label{eq:paramzero}
\mathcal{S}_n(a,b)\equiv0\pmod{\Phi_n(q)}.
\end{equation}
\end{lem}

\pf
Work in the localization of $\Q(a,b)[q]/(\Phi_n(q))$ obtained by inverting the denominator factors which occur in \eqref{eq:paramS}.  Put
\begin{equation*}
Q=q^5,
\qquad h=\frac{2n-1}{5},
\qquad m=2h.
\end{equation*}
Then $Q^h=q^{-1}$, $q=Q^{-h}$, $q^2=Q^{-2h}$, and $q^2Q^m=1$.  For $k\ge m+1$, the factor $(q^2;q^5)_k$ contains $1-q^{2+5m}=1-q^{4n}$, and hence vanishes modulo $\Phi_n(q)$.  Thus \eqref{eq:paramS} is congruent modulo $\Phi_n(q)$ to
\begin{equation*}
\sum_{k=0}^{m}T_k,
\end{equation*}
where
\begin{equation}\label{eq:Tk}
T_k=
\frac{1+qQ^k}{1+q}
\frac{(q^2,aq^2,q^2/a,bq^2,q^2/b;Q)_k}
     {(Q,aq^5,q^5/a,bq^5,q^5/b;Q)_k}Q^k .
\end{equation}

We first show the symmetry
\begin{equation}\label{eq:symmetry}
T_{m-k}=T_k\qquad(0\le k\le m).
\end{equation}
For $x$ set
\begin{equation*}
L_x(k)=\frac{(xq^2;Q)_k}{(xQ;Q)_k}.
\end{equation*}
Since $q^2Q^m=1$, reversing products gives, for $0\le k\le m$,
\begin{equation}\label{eq:Lreverse}
L_x(m-k)=L_x(m)q^{3k}L_{1/x}(k).
\end{equation}
Indeed, the omitted factors are
\begin{equation*}
\prod_{r=1}^{k}(1-xQ^{-r})=(-x)^kQ^{-k(k+1)/2}(Q/x;Q)_k
\end{equation*}
and
\begin{equation*}
\prod_{r=0}^{k-1}(1-xq^{-2}Q^{-r})=(-xq^{-2})^kQ^{-k(k-1)/2}(q^2/x;Q)_k,
\end{equation*}
which imply \eqref{eq:Lreverse}.  Furthermore,
\begin{equation}\label{eq:Lconstants}
L_x(m)L_{1/x}(m)=q^{-3m},
\qquad
L_1(m)=q^{m+1}.
\end{equation}
The first identity in \eqref{eq:Lconstants} follows because, using $q^2=Q^{-m}$, the zeros of $L_x(m)L_{1/x}(m)$ are
\begin{equation*}
Q,\ldots,Q^m,Q^{-1},\ldots,Q^{-m},
\end{equation*}
exactly the same as its poles, so the product is constant as a rational function of $x$; letting $x\to\infty$ gives $q^{-3m}$.  The second identity in \eqref{eq:Lconstants} follows from
\begin{equation*}
(q^2;Q)_m=(Q^{-m};Q)_m=(-1)^mQ^{-m(m+1)/2}(Q;Q)_m=q^{m+1}(Q;Q)_m,
\end{equation*}
because $m$ is even and $5m/2=2n-1$.

Let
\begin{equation*}
P(k)=L_a(k)L_{1/a}(k)L_b(k)L_{1/b}(k)L_1(k).
\end{equation*}
Equations \eqref{eq:Lreverse} and \eqref{eq:Lconstants} give
\begin{equation*}
P(m-k)=q^3q^{15k}P(k).
\end{equation*}
Here the constant factor is obtained as follows.  The two reciprocal
pairs \(L_aL_{1/a}\) and \(L_bL_{1/b}\) contribute \(q^{-6m}\),
while the remaining factor \(L_1\) contributes \(q^{m+1}\).  Hence the
constant part before reduction is
\begin{equation*}
q^{-6m}q^{m+1}=q^{-5m+1}=q^{-4n+3}\equiv q^3\pmod{\Phi_n(q)},
\end{equation*}
because \(5m=4n-2\).
Since
\begin{equation*}
1+qQ^{m-k}=1+q^{-1}Q^{-k}=q^{-1}Q^{-k}(1+qQ^k)
\end{equation*}
and $Q^m=q^{-2}$, the definition \eqref{eq:Tk} gives \eqref{eq:symmetry}.

It remains to prove the half-sum identity
\begin{equation}\label{eq:halfsum}
\sum_{k=0}^{h-1}T_k+\frac12T_h=0.
\end{equation}
Let $\alpha$ be an auxiliary variable and put
\begin{equation*}
c=aq^2,
\qquad
d=q^2/a,
\qquad
e=bq^2,
\qquad
f(\alpha)=\frac{\alpha^2Q^{h+1}}{cde}.
\end{equation*}
Then
\begin{equation*}
cdef(\alpha)Q^{-h}=\alpha^2Q.
\end{equation*}
Applying Lemma~\ref{lem:jackson} with base $Q$, truncation $N=h$, and
\[
x=\alpha,\qquad y=c,\qquad z=d,\qquad u=e,\qquad v=f(\alpha)
\]
gives
\begin{align}\label{eq:jacksonA}
&\sum_{k=0}^{h}
\frac{1-\alpha Q^{2k}}{1-\alpha}
\frac{(\alpha,Q^{-h},c,d,e,f(\alpha);Q)_k}
     {(Q,\alpha Q^{h+1},\alpha Q/c,\alpha Q/d,\alpha Q/e,\alpha Q/f(\alpha);Q)_k}Q^k  \notag\\
&\quad =
\frac{(\alpha Q,\alpha Q/cd,\alpha Q/ce,\alpha Q/de;Q)_h}
     {(\alpha Q/c,\alpha Q/d,\alpha Q/e,\alpha Q/cde;Q)_h}.
\end{align}
Now let $\alpha\to q^2=Q^{-2h}$.  The right-hand side of \eqref{eq:jacksonA} tends to zero because its numerator contains
\begin{equation*}
(\alpha Q/cd;Q)_h\longrightarrow(q^3;Q)_h,
\end{equation*}
and $(q^3;Q)_h$ contains $1-q^{3+5(n-3)/5}=1-q^n$.

On the left-hand side of \eqref{eq:jacksonA}, $f(\alpha)\to q^2/b$.  For $k<h$, the limit of the very-well-poised factor is
\begin{equation*}
\frac{1-q^2Q^{2k}}{1-q^2}\frac{(q;Q)_k}{(q^6;Q)_k}
=\frac{1+qQ^k}{1+q}.
\end{equation*}
For $k=h$, the factor $1-\alpha Q^{2h}$ cancels the last factor of $(\alpha Q^{h+1};Q)_h$ before the limit is taken, and hence
\begin{equation}\label{eq:middle-half}
\lim_{\alpha\to q^2}
\frac{1-\alpha Q^{2h}}{1-\alpha}\frac{(Q^{-h};Q)_h}{(\alpha Q^{h+1};Q)_h}
=\frac{1}{1+q}
=\frac12\frac{1+qQ^h}{1+q}.
\end{equation}
Therefore the limit of the left-hand side of \eqref{eq:jacksonA} is exactly the left-hand side of \eqref{eq:halfsum}.  This proves \eqref{eq:halfsum}.  By the symmetry \eqref{eq:symmetry},
\begin{equation*}
\sum_{k=0}^{m}T_k=2\sum_{k=0}^{h-1}T_k+T_h=0.
\end{equation*}
Thus \eqref{eq:paramzero} follows. \qed

For $x$ an indeterminate, put
\begin{equation*}
R_x(q)=
\frac{(xq^4,q^4/x,q^3,q^7;q^5)_m}
     {(xq^5,q^5/x,q^2,q^6;q^5)_m},
\qquad m=\frac{4n-2}{5},
\end{equation*}
and set
\begin{equation*}
P_a=(1-aq^{4n})(a-q^{4n}),
\qquad
P_b=(1-bq^{4n})(b-q^{4n}).
\end{equation*}

\begin{lem}\label{lem:paramcong}
Let $n\equiv3\pmod5$ and $m=(4n-2)/5$.  Then
\begin{align}\label{eq:paramcong}
\mathcal{S}_n(a,b)
&\equiv
\frac{(1-bq^{4n})(b-q^{4n})(-1-a^2+aq^{4n})}{(a-b)(1-ab)}R_b(q)\notag\\
&\quad+
\frac{(1-aq^{4n})(a-q^{4n})(-1-b^2+bq^{4n})}{(b-a)(1-ba)}R_a(q)
\pmod{\Phi_n(q)P_aP_b}.
\end{align}
\end{lem}

\pf
First recall the standard microscoping congruence modulo $P_aP_b$ in Guo~\cite{Guo2025}.  If $a=q^{4n}$ or $a=q^{-4n}$, then one of the two parameters $aq^2$ and $q^2/a$ becomes $q^{2+4n}$ and the other becomes $q^{2-4n}=q^{-5m}$.  Jackson's summation \eqref{eq:jackson} gives
\begin{equation*}
\mathcal{S}_n(a,b)=R_b(q)
\end{equation*}
for these two specializations.  Hence
\begin{equation*}
\mathcal{S}_n(a,b)\equiv R_b(q)\pmod{P_a}.
\end{equation*}
By symmetry in $a$ and $b$,
\begin{equation*}
\mathcal{S}_n(a,b)\equiv R_a(q)\pmod{P_b}.
\end{equation*}
The two rational factors in \eqref{eq:paramcong} are the corresponding Chinese-remainder interpolation factors.  We use the usual
creative-microscoping localization in which the factors \((a-b)(1-ab)\)
and \((b-a)(1-ba)\) are units; equivalently, one may clear these
denominators before checking the congruence in the polynomial ring.  Directly,
\begin{align*}
\frac{(1-bq^{4n})(b-q^{4n})(-1-a^2+aq^{4n})}{(a-b)(1-ab)}&\equiv1\pmod{P_a},\\
\frac{(1-aq^{4n})(a-q^{4n})(-1-b^2+bq^{4n})}{(b-a)(1-ba)}&\equiv1\pmod{P_b},
\end{align*}
and each factor is congruent to zero modulo the other parameter modulus.
For instance, the first congruence modulo \(P_a=(1-aq^{4n})(a-q^{4n})\)
is checked by substituting the two roots \(a=q^{4n}\) and
\(a=q^{-4n}\); the second is identical with \(a\) and \(b\) interchanged.
This proves \eqref{eq:paramcong} modulo $P_aP_b$.

It remains to prove \eqref{eq:paramcong} modulo $\Phi_n(q)$.  The left-hand side is congruent to zero by Lemma~\ref{lem:paramzero}.  We show that the right-hand side is also congruent to zero.  Write $n=5s+3$, so $m=4s+2$.  In $R_x(q)$ the numerator contains
\begin{equation*}
1-q^n\quad\hbox{from }(q^3;q^5)_m,
\qquad
1-q^{4n}\quad\hbox{from }(q^7;q^5)_m,
\end{equation*}
whereas the only factor independent of $x$ in the denominator divisible by $\Phi_n(q)$ is
\begin{equation*}
1-q^{2n}\quad\hbox{from }(q^6;q^5)_m.
\end{equation*}
Consequently $R_x(q)$ contains
\begin{equation*}
\frac{(1-q^n)(1-q^{4n})}{1-q^{2n}}
=(1-q^n)\frac{1+q^n+q^{2n}+q^{3n}}{1+q^n}
\end{equation*}
times a unit in the localization at $\Phi_n(q)$.  Thus $R_x(q)\equiv0\pmod{\Phi_n(q)}$.  Since $\Phi_n(q)$, $P_a$, and $P_b$ are pairwise coprime, the Chinese remainder theorem gives \eqref{eq:paramcong}. \qed

We also need the following divisibility observation.

\begin{lem}\label{lem:Aorder}
Let $n=5s+3$ and $m=4s+2$.  Put
\begin{equation*}
A_n(q)=
\frac{(q^3,q^4,q^4,q^7;q^5)_m}
     {(q^2,q^5,q^5,q^6;q^5)_m}.
\end{equation*}
Then $A_n(q)$ is divisible by $\Phi_n(q)^3$.
\end{lem}

\pf
The numerator of $A_n(q)$ contains
\begin{equation*}
1-q^n\quad\hbox{from }(q^3;q^5)_m,
\end{equation*}
\begin{equation*}
1-q^{3n}\quad\hbox{from each of the two factors }(q^4;q^5)_m,
\end{equation*}
and
\begin{equation*}
1-q^{4n}\quad\hbox{from }(q^7;q^5)_m.
\end{equation*}
The denominator contains one factor divisible by $\Phi_n(q)$, namely
\begin{equation*}
1-q^{2n}\quad\hbox{from }(q^6;q^5)_m,
\end{equation*}
and all remaining denominator factors are units modulo $\Phi_n(q)$.  Therefore, up to a unit at $\Phi_n(q)$,
\begin{equation*}
A_n(q)=
\frac{(1-q^n)(1-q^{3n})^2(1-q^{4n})}{1-q^{2n}}\,U(q)
\end{equation*}
for some unit $U(q)$.  Since
\begin{equation*}
\frac{(1-q^n)(1-q^{3n})^2(1-q^{4n})}{1-q^{2n}}
=(1-q^n)^3
\frac{(1+q^n+q^{2n})^2(1+q^n+q^{2n}+q^{3n})}{1+q^n},
\end{equation*}
the claim follows. \qed

\noindent{\it Proof of Theorem~\ref{thm:conj61}.}
Put
\begin{equation*}
X=q^{4n},
\qquad
A(a)=\frac{(aq^4,q^4/a,q^3,q^7;q^5)_m}{(aq^5,q^5/a,q^2,q^6;q^5)_m},
\qquad
A_0=A(1)=A_n(q).
\end{equation*}
In Lemma~\ref{lem:paramcong}, let $b=1$.  Since $(1-q^{4n})^2=(1-X)^2$ contributes two powers of $\Phi_n(q)$, we obtain, modulo $\Phi_n(q)^3P_a$,
\begin{align}\label{eq:afterb}
\mathcal{S}_n(a,1)
&\equiv
\frac{(1-X)^2(1+a^2-aX)}{(1-a)^2}A_0
-
\frac{(1-aX)(a-X)(2-X)}{(1-a)^2}A(a)\notag\\
&=(1-X)^2A_0
+(2-X)\left\{
\frac{a(1-X)^2}{(1-a)^2}A_0
-
\frac{(1-aX)(a-X)}{(1-a)^2}A(a)
\right\}.
\end{align}
The expression in braces in \eqref{eq:afterb} has a removable
singularity at \(a=1\).  We verify this and compute the limiting value.

Let
\begin{equation*}
H_m(q)=\sum_{j=1}^{m}\left(\frac{q^{5j-1}}{[5j-1]^2}-\frac{q^{5j}}{[5j]^2}\right).
\end{equation*}
The required limiting value is
\begin{equation}\label{eq:limit}
\lim_{a\to1}\left\{
\frac{a(1-X)^2}{(1-a)^2}A_0
-
\frac{(1-aX)(a-X)}{(1-a)^2}A(a)
\right\}
=A_0\left(X+[4n]^2H_m(q)\right).
\end{equation}
To verify \eqref{eq:limit}, write $A(a)=A_0G(a)$.  Then $G(1)=1$ and $G'(1)=0$ by the symmetry $G(a)=G(1/a)$.  Moreover
\begin{equation*}
\left.\frac{d^2}{da^2}\log\bigl((aq^r,q^r/a;q^5)_m\bigr)\right|_{a=1}
=-2\sum_{j=0}^{m-1}\frac{q^{r+5j}}{(1-q^{r+5j})^2},
\end{equation*}
and hence
\begin{equation*}
G''(1)=\left.(\log G(a))''\right|_{a=1}
=-\frac{2}{(1-q)^2}H_m(q).
\end{equation*}
Expanding
\begin{equation*}
G(a)=1+\frac{G''(1)}{2}(a-1)^2+O((a-1)^3)
\end{equation*}
gives \eqref{eq:limit}.  This calculation also shows explicitly that
the apparent pole of order two at \(a=1\) in \eqref{eq:afterb} is
removable.  Hence it is legitimate to specialize \(a=1\).  Under this
specialization
\begin{equation*}
P_a=(1-aq^{4n})(a-q^{4n})\longmapsto (1-q^{4n})^2=(1-X)^2,
\end{equation*}
which contributes two additional factors of \(\Phi_n(q)\).  Therefore
the congruence modulo \(\Phi_n(q)^3P_a\) specializes to a congruence
modulo \(\Phi_n(q)^3(1-X)^2\), and hence modulo \(\Phi_n(q)^5\).

It follows from \eqref{eq:afterb} and \eqref{eq:limit} that, modulo $\Phi_n(q)^5$,
\begin{align}\label{eq:withH}
\sum_{k=0}^{n-1}\frac{(1+q^{5k+1})(q^2;q^5)_k^5}{(1+q)(q^5;q^5)_k^5}q^{5k}
&\equiv A_0\left\{(1-X)^2+(2-X)\left(X+[4n]^2H_m(q)\right)\right\}\notag\\
&=A_0\left\{1+(2-X)[4n]^2H_m(q)\right\}.
\end{align}
Furthermore,
\begin{equation*}
[4n]^2\bigl(2-X-X^{-1}\bigr)=-X^{-1}[4n]^2(X-1)^2
\end{equation*}
is divisible by $\Phi_n(q)^4$.  The expression $H_m(q)$ has at most a double pole at $\Phi_n(q)$:
the only possible pole comes from the denominator \([5j-1]^2\) when
\(5j-1\) is a multiple of \(n\), which occurs for
\(j=(3n+1)/5\).  Meanwhile \(A_0\) is divisible by $\Phi_n(q)^3$ by
Lemma~\ref{lem:Aorder}.  Therefore the error caused by replacing $2-X$
by $X^{-1}=q^{-4n}$ in \eqref{eq:withH} is divisible by $\Phi_n(q)^5$.  Thus
\begin{equation}\label{eq:almost}
\sum_{k=0}^{n-1}\frac{(1+q^{5k+1})(q^2;q^5)_k^5}{(1+q)(q^5;q^5)_k^5}q^{5k}
\equiv A_0\left(1+[4n]^2q^{-4n}H_m(q)\right)
\pmod{\Phi_n(q)^5}.
\end{equation}

Finally, the only term in $H_m(q)$ whose denominator is divisible by $\Phi_n(q)$ is the one with
\begin{equation*}
j_0=\frac{3n+1}{5},
\qquad
5j_0-1=3n.
\end{equation*}
Therefore
\begin{equation*}
H_m(q)=\frac{q^{3n}}{[3n]^2}+H_m^*(q),
\end{equation*}
where $H_m^*(q)$ is regular at $\Phi_n(q)$.  By Lemma~\ref{lem:Aorder}, the product
\begin{equation*}
A_0[4n]^2q^{-4n}H_m^*(q)
\end{equation*}
is divisible by $\Phi_n(q)^5$.  Hence \eqref{eq:almost} becomes
\begin{equation*}
\sum_{k=0}^{n-1}\frac{(1+q^{5k+1})(q^2;q^5)_k^5}{(1+q)(q^5;q^5)_k^5}q^{5k}
\equiv
A_0\left(1+\frac{[4n]^2q^{-n}}{[3n]^2}\right)
\pmod{\Phi_n(q)^5},
\end{equation*}
which is exactly \eqref{eq:theorem-one}. \qed

\section{Proof of Theorem \ref{thm:conj62}}\label{sec:proof-two}

Recall the standard definition of a basic hypergeometric series, as in Gasper and Rahman~\cite{GasperRahman}:
\begin{equation*}
{}_{r+1}\phi_r\!
\left[\begin{matrix}a_1,\ldots,a_{r+1}\\ b_1,\ldots,b_r\end{matrix};q,z\right]
=
\sum_{k=0}^{\infty}
\frac{(a_1,\ldots,a_{r+1};q)_k}{(q,b_1,\ldots,b_r;q)_k}z^k.
\end{equation*}
We shall use the following terminating ${}_4\phi_3$ summation of Andrews~\cite{Andrews2011} (see also Guo~\cite{Guo2013Andrews}).

\begin{lem}\label{lem:andrews}
For every integer $N\ge0$,
\begin{equation}\label{eq:andrews}
{}_{4}\phi_{3}\!\left[
\begin{matrix}
 q^{-2N},\ a,\ b,\ q^{1-2N}/ab\\
 q^{2-2N}/a,\ q^{2-2N}/b,\ abq
\end{matrix};q^2,q^2\right]
=
q^{-N}\frac{(a,b,-q;q)_N(ab;q^2)_N}
{(ab;q)_N(a,b;q^2)_N}.
\end{equation}
\end{lem}

We need one elementary product reduction.

\begin{lem}\label{lem:product}
Let $n=4m+1$ with $m\ge1$.  Then
\begin{equation}\label{eq:A-B}
(q^4;q^4)_m\equiv (-1)^m q^{2m(m+1)}(q;q^4)_m\pmod{\Phi_n(q)}.
\end{equation}
Consequently,
\begin{equation}\label{eq:ratio-lemma}
\frac{(q^4;q^4)_m}{(q;q^4)_m^2}
\equiv
\frac{q^{4m(m+1)}}{(q^4;q^4)_m}
\pmod{\Phi_n(q)}.
\end{equation}
\end{lem}

\pf
Since $q^n\equiv1\pmod{\Phi_n(q)}$ and $n=4m+1$, for $1\le j\le m$ we have
\begin{equation}\label{eq:factor-transform}
1-q^{4j}=-q^{4j}(1-q^{-4j})
\equiv -q^{4j}(1-q^{n-4j})
=-q^{4j}(1-q^{4(m-j)+1})
\pmod{\Phi_n(q)}.
\end{equation}
Multiplying \eqref{eq:factor-transform} over $j=1,\ldots,m$ gives
\begin{equation*}
(q^4;q^4)_m
\equiv
(-1)^m q^{4(1+2+\cdots+m)}
\prod_{j=1}^{m}(1-q^{4(m-j)+1})
=(-1)^m q^{2m(m+1)}(q;q^4)_m,
\end{equation*}
which proves \eqref{eq:A-B}.

The factors in $(q;q^4)_m$ and $(q^4;q^4)_m$ are coprime to $\Phi_n(q)$.  Indeed, $\Phi_n(q)$ divides $1-q^s$ only when $n\mid s$, and the exponents $1,5,\ldots,4m-3$ and $4,8,\ldots,4m$ are not divisible by $n=4m+1$.  Hence these two products are invertible modulo $\Phi_n(q)$.  From \eqref{eq:A-B}, if
\begin{equation*}
A=(q^4;q^4)_m,
\qquad
B=(q;q^4)_m,
\qquad
C=(-1)^m q^{2m(m+1)},
\end{equation*}
then $A\equiv CB$, and therefore
\begin{equation*}
\frac{A}{B^2}\equiv\frac{C^2}{A}
=\frac{q^{4m(m+1)}}{(q^4;q^4)_m}
\pmod{\Phi_n(q)}.
\end{equation*}
This is \eqref{eq:ratio-lemma}. \qed

\noindent{\it Proof of Theorem~\ref{thm:conj62}.}
Put
\begin{equation*}
m=\frac{n-1}{4},
\qquad\hbox{so that}\qquad n=4m+1.
\end{equation*}
Because $n$ is odd and $\gcd(n,4)=1$, the factors of $(q^4;q^4)_k$ are coprime to $\Phi_n(q)$ for $0\le k\le n-1$.  Thus the summands in \eqref{eq:theorem-two} are well-defined modulo $\Phi_n(q)$.  For $k\ge m+1$, the product $(q;q^4)_k$ contains the factor $1-q^{1+4m}=1-q^n$, so the corresponding summand is congruent to zero modulo $\Phi_n(q)$.  Hence
\begin{equation}\label{eq:truncate}
\sum_{k=0}^{n-1}\frac{(q;q^4)_k^4}{(q^4;q^4)_k^4}q^{4k}
\equiv
\sum_{k=0}^{m}\frac{(q;q^4)_k^4}{(q^4;q^4)_k^4}q^{4k}
\pmod{\Phi_n(q)}.
\end{equation}

In Andrews' summation \eqref{eq:andrews}, make the substitution
\begin{equation}\label{eq:specialization-andrews}
N=m,
\qquad
q\to q^2,
\qquad
a=b=q.
\end{equation}
The left-hand side becomes
\begin{equation}\label{eq:andrews-lhs-specialized}
{}_{4}\phi_{3}\!\left[
\begin{matrix}
 q^{-4m},\ q,\ q,\ q^{-4m}\\
 q^{3-4m},\ q^{3-4m},\ q^4
\end{matrix};q^4,q^4\right].
\end{equation}
It terminates at $k=m$.  Since $q^n\equiv1\pmod{\Phi_n(q)}$ and $n=4m+1$, we have, termwise for $0\le k\le m$,
\begin{equation*}
(q^{-4m};q^4)_k\equiv(q;q^4)_k,
\qquad
(q^{3-4m};q^4)_k\equiv(q^4;q^4)_k
\pmod{\Phi_n(q)}.
\end{equation*}
All denominator factors in this reduction are coprime to $\Phi_n(q)$.  Therefore \eqref{eq:andrews-lhs-specialized} is congruent to
\begin{equation}\label{eq:lhs-reduced}
\sum_{k=0}^{m}\frac{(q;q^4)_k^4}{(q^4;q^4)_k^4}q^{4k}
\pmod{\Phi_n(q)}.
\end{equation}

The right-hand side of \eqref{eq:andrews} under the specialization \eqref{eq:specialization-andrews} is
\begin{equation}\label{eq:rhs-raw}
q^{-2m}
\frac{(q;q^2)_m^2(-q^2;q^2)_m(q^2;q^4)_m}
{(q^2;q^2)_m(q;q^4)_m^2}.
\end{equation}
The denominator in \eqref{eq:rhs-raw} is invertible modulo $\Phi_n(q)$: for $(q^2;q^2)_m$ the exponents are $2,4,\ldots,2m<n$, and for $(q;q^4)_m$ this was noted in Lemma~\ref{lem:product}.  Using the elementary identity
\begin{equation*}
(-q^2;q^2)_m(q^2;q^2)_m=(q^4;q^4)_m,
\end{equation*}
we rewrite \eqref{eq:rhs-raw} as
\begin{equation}\label{eq:rhs-middle}
q^{-2m}
\frac{(q;q^2)_m^2(q^2;q^4)_m(q^4;q^4)_m}
{(q^2;q^2)_m^2(q;q^4)_m^2}.
\end{equation}
By Lemma~\ref{lem:product},
\begin{equation*}
\frac{(q^4;q^4)_m}{(q;q^4)_m^2}
\equiv
\frac{q^{4m(m+1)}}{(q^4;q^4)_m}
\pmod{\Phi_n(q)}.
\end{equation*}
Consequently,
\begin{equation}\label{eq:rhs-reduced}
q^{-2m}
\frac{(q;q^2)_m^2(q^2;q^4)_m(q^4;q^4)_m}
{(q^2;q^2)_m^2(q;q^4)_m^2}
\equiv
q^{-2m+4m(m+1)}
\frac{(q;q^2)_m^2(q^2;q^4)_m}
{(q^2;q^2)_m^2(q^4;q^4)_m}
\pmod{\Phi_n(q)}.
\end{equation}
Finally,
\begin{equation*}
-2m+4m(m+1)-m=m(4m+1)=mn,
\end{equation*}
so
\begin{equation}\label{eq:q-power-reduction}
q^{-2m+4m(m+1)}\equiv q^m\pmod{\Phi_n(q)}.
\end{equation}
Combining \eqref{eq:truncate}, \eqref{eq:lhs-reduced}, and \eqref{eq:rhs-reduced}, and then using \eqref{eq:q-power-reduction}, yields
\begin{equation}\label{eq:final-theorem-two}
\sum_{k=0}^{n-1}\frac{(q;q^4)_k^4}{(q^4;q^4)_k^4}q^{4k}
\equiv
q^m\frac{(q;q^2)_m^2(q^2;q^4)_m}
{(q^2;q^2)_m^2(q^4;q^4)_m}
\pmod{\Phi_n(q)}.
\end{equation}
Since $m=(n-1)/4$, \eqref{eq:final-theorem-two} is exactly \eqref{eq:theorem-two}. \qed

\end{document}